\documentclass[reqno, 12pt]{amsart}

\usepackage{amsfonts, amsthm, amsmath, amssymb}
\usepackage{hyperref}
\hypersetup{colorlinks=false}

\usepackage[margin=3.5cm]{geometry}

\usepackage{graphicx}

\RequirePackage{mathrsfs} \let\mathcal\mathscr

\numberwithin{equation}{section}

\usepackage[utf8]{inputenc}
\newtheorem{thm}{Theorem}[section]
\newtheorem{lem}[thm]{Lemma}

\def\R{\mathbb{R}}
\def\Q{\mathbb{Q}}
\def\Z{\mathbb{Z}}
\def\F{\mathbb{F}}
\def\PP{\mathbb{P}}
\newcommand{\ZZp}{\mathbb{Z}_{\mathrm{prim}}}

\def\mmod{\!\!\!\mod}
\def\tilde{\widetilde}

\DeclareMathOperator{\Gal}{Gal}
\DeclareMathOperator{\Spec}{Spec}
\DeclareMathOperator{\Proj}{Proj}

\theoremstyle{remark}
\newtheorem*{ack}{Acknowledgement}
\newtheorem{defi}[thm]{Definition}
\newtheorem{rema}{Remark}[section]

\def\a{\mathbf{y}}
\def\y{\mathbf{y}}

\def\b{\mathbf{t}}

\def\u{\mathbf{u}}

\begin{document}

\title[Local solubility for a family of quadrics]{Local solubility for a family of quadrics over a split quadric surface}

\author{Tim Browning}
\author{Julian Lyczak}
\author{Roman Sarapin}

\address{IST Austria\\
Am Campus 1\\
3400 Klosterneuburg\\
Austria}
\email{tdb@ist.ac.at}
\email{jlyczak@ist.ac.at}
\email{romansarapin13@gmail.com}

\subjclass[2010]
{14G05  
(11N36, 
14D10) 
}

\begin{abstract}
We study  the density of 
everywhere locally soluble diagonal quadric surfaces, parameterised by rational points that lie on a split quadric surface.
\end{abstract}

\date{\today}

\maketitle
\thispagestyle{empty}

\section{Introduction}

A great deal of recent activity has been directed at the question of 
local solubility for families of varieties.  
Let $X$ and $Y$
 be  smooth, proper and geometrically irreducible varieties over  $\Q$, equipped with a dominant morphism $\pi \colon X\to Y$,  with geometrically integral generic fibre.  Suppose  that 
 $Y$ is Fano and equipped with an anticanonical height function 
 $H \colon Y(\Q)\to \R$.  This short note is concerned with  the behaviour of the counting function
$$
N_{\text{loc}}(\pi,B)=\#\left\{
y\in Y(\Q)\cap \pi(X(\mathbf{A}_\Q)) \colon H(y)\leq B
\right\},
$$
as $B\to \infty$, 
where $\mathbf{A}_\Q$ denotes the ring of ad\`eles.   We shall consider an explicit example, with  $Y\subset \PP^3$  a split quadric surface,
which calls into question our expectations about the behaviour of $N_{\text{loc}}(\pi,B)$.

The quantity $N_{\text{loc}}(\pi,B)$  measures the density of fibres of $\pi$ which are everywhere locally soluble, and it has received the most attention  when $Y$ is projective space. 
For example, it follows from work of 
 Poonen and Voloch \cite[Thm.~3.6]{PV} that 
\begin{equation}\label{eq:PV}
N_{\text{loc}}(\pi,B)\sim c_{d,n} B,
\end{equation}
as $B\to \infty$, for an explicit constant $c_{d,n}>0$, when $\pi \colon X\to \PP^N$ is the   family of degree $d$ hypersurfaces in $\PP^n$ over $\Q$, with  $(d,n)\neq (2,2)$ and $N=\binom{n+d}{d}-1$. 
Moreover, when $Y=\PP^n$, 
work of Loughran and Smeets \cite{LS} shows that 
\begin{equation}\label{eq:upper}
N_{\text{loc}}(\pi,B)\ll \frac{B}{(\log B)^{\Delta(\pi)}},
\end{equation}
for a certain quantity $\Delta(\pi)\geq 0$, whose definition will be recalled in 
\S \ref{s:geom} and which can be extended to arbitrary fibrations $\pi \colon X\to Y$. 
For $Y=\PP^n$, it is   conjectured in \cite{LS}  that the upper bound
\eqref{eq:upper} should be sharp whenever the fibre 
$\pi^{-1}(D)$
above each codimension one point  $D$ in $\PP^n$ contains a component of multiplicity one.

It is natural to ask   what  happens when  $Y$ is a Fano variety over $\Q$ that is 
not projective space. We shall assume that $Y(\Q)$ is Zariski dense in $Y$ and that $Y$ 
conforms to the version of the Manin conjecture \cite{f-m-t} without   accummulating subvarieties, so that 
$$
\#\{y\in Y(\Q) \colon H(y)\leq B\} \sim c_Y B(\log B)^{\rho_Y-1},
$$
as $B\to \infty$, where $c_Y>0$ is a constant and $\rho_Y$ is the rank of the Picard group of $Y$. 
Assuming that $Y$ is a Fano variety over $\Q$ without   accummulating subvarieties, 
a naive combination of the Manin  conjecture with the upper bound \eqref{eq:upper} might lead us to suppose that 
\begin{equation}\label{eq:guess}
N_{\text{loc}}(\pi,B)\sim  \frac{c_\pi B (\log B)^{\rho_Y-1}}{(\log B)^{\Delta(\pi)}},
\end{equation}
as $B\to \infty$, for a suitable constant $c_\pi>0$.  This is compatible with \eqref{eq:PV}, since in this case $\Delta(\pi)=0$ and   $\rho_{\PP^N}=1$.
It also compatible with an upper bound of
Browning and Loughran \cite[Thm.~1.10]{BL},
which concerns the case where  $Y\subset \PP^n$ is a smooth quadric over $\Q$ of dimension at least $3$, since then $\rho_Y=1$.

We  study an explicit example that 
addresses a case where the variety $Y$ has Picard number $\rho_Y>1$ and which contradicts the expectation  \eqref{eq:guess}. Let 
$Y\subset \mathbb P^3$ be the split 
quadric surface
 \begin{equation}\label{eq:Y}
 y_0y_1=y_2y_3.
 \end{equation}
Clearly $Y$ is smooth and  $\rho_Y=2$.
Let  $X\subset \mathbb P^3\times \mathbb P^3$ be the variety defined by the pair of equations
\begin{equation}\label{eq:X}
\begin{split}
y_0x_0^2+y_1x_1^2+y_2x_2^2+y_3x_3^2&=0\\
y_0y_1&=y_2y_3.
\end{split}
\end{equation}
Then 
$X$ is equipped with a dominant morphism $\pi\colon X\to Y$ whose generic fibre is geometrically irreducible. We shall see in \S \ref{s:geom} that the total space $X$ is singular  and so we let  $\tilde X\to X $ be a desingularisation. Taking 
$\tilde \pi$ to be the composition with $\pi\colon X\to Y$, we shall prove the following bounds 
for the associated counting function $N_{\text{loc}}(\tilde \pi,B)$ in \S \ref{s:ant}.

\begin{thm}\label{t}
We have 
$
B\ll N_{\text{loc}}(\tilde \pi,B)\ll B.
$
\end{thm}

Since the Hasse principle holds for quadratic forms, we deduce from Theorem~\ref{t}  that 
the same upper and lower bounds hold for the quantity 
$ N_{\text{glob}}(\tilde \pi,B)$,  counting the number of 
$y\in Y(\Q)\cap \pi(\tilde X(\Q))$ such that $H(y)\leq B$.

The  next result, proved in \S \ref{s:geom},  shows that 
our  example is not compatible with the naive expectation 
\eqref{eq:guess}, since 
 $\rho_Y=2$.

\begin{thm}\label{t'}
We have $\Delta(\tilde \pi)=2$.
\end{thm}

Just as thin sets have been found to have a large influence on modern formulations of the Manin conjecture, as discussed by Peyre \cite{peyre-freedom},
so we expect that the naive expectation
\eqref{eq:guess} should hold after the removal of appropriate thin sets from $Y(\Q)$. 
In fact, for our example \eqref{eq:X}, we conjecture that 
$$
\#\left\{
y\in Y(\Q) \cap \pi(\tilde X(\mathbf{A}_\Q)) \colon H(y)\leq B, ~-y_0y_2\neq \square, ~
-y_0y_3\neq \square
\right\} \sim c \frac{B}{\log B},
$$
for an appropriate  constant $c>0$. 
It seems plausible that the character sum method developed by Friedlander and Iwaniec \cite{fi} will prove useful  for tackling this counting problem, although the hyperboloid height conditions that arise from parameterising the points in $Y(\mathbb{Q})$ will undoubtedly complicate the argument. 

In this note we have focused attention on the quadric surface bundle $X\to Y$, where $X$ is given by \eqref{eq:X}.  One can also study the conic bundle $\eta \colon W\to Y$, where $W\subset \PP^2\times \PP^3$ is given by
\begin{align*}
y_0x_0^2+y_1x_1^2+y_2x_2^2&=0\\
y_0y_1&=y_2y_3.
\end{align*}
This variety is singular and so we consider the morphism $\tilde\eta \colon\tilde W\to Y$, where $\tilde \eta$ is the composition of a desingularisation
$\tilde W\to W$ with $\eta$. 
The methods of this paper carry over to this case with little adjustment and can be used to prove that 
$B\ll N_{\text{loc}}(\tilde \eta,B)\ll B$. Moreover, one also has $\Delta(\tilde \eta)=2$ in this case, providing a further example that  counters the expectation \eqref{eq:guess}, but one which we believe can be repaired via the removal of appropriate thin sets. 
(In fact, for $i\in \{0,1\}$,  the fibre of $\eta$ above the line $y_i=y_2=0$ has multiplicity two in $W$, whereas the corresponding fibre of $\tilde \eta$ admits  an additional component of multiplicity one.)

\begin{ack}
While working on this paper the first  author was
supported by FWF grant P~32428-N35.  This paper was written while the third author was an intern at IST. 
\end{ack}

\section{Geometric input}\label{s:geom}

We begin by extending the definition 
of 
 $\Delta(\pi)$ 
from \cite{LS} to general fibrations.  
 Let $X$ and $Y$
 be arbitrary  
 proper and geometrically irreducible varieties over  $\Q$, equipped with a  
 dominant morphism $\pi \colon X\to Y$, with geometrically
integral generic fibre. Assume that $Y$ is smooth and 
associate the residue field 
 $\kappa(D)$ to any codimension one point  $D\in Y^{(1)}$. Let 
 $S_{D}$ be the set of geometrically irreducible components of $\pi^{-1}(D)$ and 
 choose a finite group $\Gamma_D$ through which the action of $\Gal(\overline{\kappa(D)}/\kappa(D))$ on $S_{D}$ factors.  Then $\Delta(\pi)$ is defined to be 
\begin{equation}\label{eq:def_Delta}
\Delta(\pi)
=\sum_{D\in Y^{(1)}} \left(1-\delta_{D}(\pi)\right),
\end{equation}
where
$$
\delta_{D}(\pi)=
\frac{\#\{
\sigma\in \Gamma_D \colon   \text{$\sigma$ acts with a fixed point on 
$S_{D}$}\}}
{\#\Gamma_D}.
$$

Henceforth, we take $Y\subset \PP^3$ to be the split quadric \eqref{eq:Y} and $X\subset \PP^3\times\PP^3$ to be the variety \eqref{eq:X}. 
Consider the fibration $\pi \colon X\to Y$. There are precisely $4$ codimension one points in $ Y $ which produce reducible fibres. These are the four lines 
\begin{equation}\label{eq:lines}
D_{i,j}=\{y_i=y_j=0\}, \quad \text{for $i\in \{0,1\}$,  $j\in \{2,3\}$}.
\end{equation} 
Above each of these the fibre of $\pi$ is split by a quadratic extension, so that $\delta_{D_{i,j}}(\pi)=\frac{1}{2}$. It follows that $\Delta(\pi)=4\cdot \frac{1}{2}=2$
in \eqref{eq:def_Delta}. 

However, the Jacobian of $X$ is 
$$ 
J=\begin{pmatrix}
x_0^2 & x_1^2 & x_2^2 & x_3^2 & 2y_0x_0 & 2y_1x_1 & 2y_2x_2 & 2y_3x_3\\
y_1  &     y_0 &   -y_3    &  -y_2      &  0 &            0  &              0 &             0
\end{pmatrix}.
$$ 
Thus $X$ is singular, with  singular locus  supported on the union of subschemes 
$$
\Sigma_{i,j} = V(y_i,y_j,x_{i'},x_{j'},y_{j'}x_i^2+y_{i'}x_j^2), 
\quad
\text{ for $\{i,i'\}=\{0,1\}$ and $\{j,j'\}=\{2,3\}$},
$$
where we note that a choice of $i$ and $j$ uniquely determines $i'$ and $j'$. In fact, away from the intersection of two such subschemes, the singular locus $\Sigma$ is precisely the union of these four reduced subschemes.
Let $\tilde X \to X$ be a desingularisation of $X$ and let $\tilde \pi$ be the composition with $\pi \colon X \to Y$. The main aim of this section is to prove that 
\begin{equation}\label{eq:goal}
\Delta(\tilde \pi)=2,
\end{equation}
as claimed in
Theorem \ref{t'}.

We begin by establishing the following geometric lemma that will be useful in our treatment. 

\begin{lem}\label{lem:spread}
Let $C$ and $C'$ be two integral curves smooth and projective over a field $k$. Let us write $\kappa$ and $\kappa'$ for their function fields. Consider a proper morphism $$\rho \colon W \to C \times C'.
$$ 
If $W\times C'_{\kappa}$ is smooth over $\kappa$ and 
 $W\times {C}_{\kappa'}$ is smooth over $\kappa'$,  then the image of the singular locus of $W/k$ under $\rho$ has dimension $0$.
\end{lem}

\begin{proof}
Consider the composite morphism $$W \to C\times_k C' \to C,$$
which has a smooth generic fibre by assumption. By spreading out \cite[Thm.~3.2.1]{Poonen}, we see that $W_U \to U$ is smooth for a dense open subset $U \subset C$. We conclude that $W_U \to U \to \Spec k$ is smooth. Similarly, we have a $V \subset C'$ such that $W_V \to V$ is smooth. We conclude that the  image of the smooth locus of $W \to \Spec k$ under $\rho$ contains both $U\times C'$ and $C\times V$.
\end{proof}

It suffices to prove \eqref{eq:goal} for a single desingularisation.

\begin{defi}
Let $X' \to X$ be the blowup in the union of the closed subschemes $\Sigma_{i,j}$. Let $\tilde X \to X'$ be a fixed strong desingularisation.
\end{defi}

We have seen that the image of the singular locus $\Sigma$ under $X \to Y$ is one-dimensional and supported on the union of lines $D_{i,j} \subset Y$ in \eqref{eq:lines}. 
Let $\eta_{i,j}$ be the generic point of the line $D_{i,j}.$
We will prove that the image of the singular locus of  $X'$ under $\pi' \colon X' \to Y$ has  lower dimension.

\begin{lem}
The image of the singular locus $\Sigma'$ of $X'$ under $X' \to Y$ is (at most) zero-dimensional. The fibre $X'_{\eta_{i,j}}$ has four geometrically irreducible components, of which  two  make up the strict transform of $X_{\eta_{i,j}}$ and two come from the exceptional divisor of the blowup $X'\to X$.
These components are conjugated in pairs.
\end{lem}

\begin{proof}
We identify $Y$ with $\mathbb P^1 \times \mathbb P^1$, where the projective lines have coordinates $u=\frac{y_i}{y_j} = \frac{y_{j'}}{y_{i'}}$ and $t=\frac{y_i}{y_{j'}} = \frac{y_j}{y_{i'}}$, respectively. By Lemma \ref{lem:spread} and symmetry we only need to show that $X' \times_Y \mathbb P^1_{\mathbb Q(u)}$ is smooth over $\mathbb P^1_{\mathbb Q(u)}$. Again, by symmetry, we can restrict to the affine opens where $t\ne \infty$.

Since blowing up commutes with pulling back along flat morphisms we find that $X' \times_Y \mathbb A^1_{\mathbb Q(u)}$ and the blowup of $X \times_Y \mathbb A^1_{\mathbb Q(u)}$ in the pullback of $\Sigma$ are naturally isomorphic over $\mathbb A^1_{\mathbb Q(u)}$. The three subschemes $\Sigma_{i',j}$, $\Sigma_{i,j'}$ and $\Sigma_{i',j'}$ pull back to the empty scheme under $X \times _Y \mathbb A^1_{\mathbb Q(u)} \to X$ and it will therefore suffice to compute the blowup of $X \times _Y \mathbb A^1_{\mathbb Q(u)}$ in the pullback of $\Sigma_{i,j}$.

We see that $X \times _Y \mathbb A^1_{\mathbb Q(u)} \subset \mathbb P^3 \times_{\mathbb Q} \mathbb A^1_{\mathbb Q(u)}$ is given by $ut x_i^2+x_{i'}^2 + tx_j^2 + ux_{j'}^2=0$. The pullback of $\Sigma_{i,j}$ equals $V(t,x_{i'},x_{j'},ux_i^2+x_j^2)$, which is contained in the affine open given by $x_i \ne 0$. By abuse of notation we will use the same notation for the variables on this affine open.

We can compute the blowup $Z$ of $\{ut+x_{i'}^2 + tx_j^2 + ux_{j'}^2=0\} \subset \mathbb A^3 \times_{\mathbb Q} \mathbb A^1_{\mathbb Q(u)}$ in the subscheme $V(t,x_{i'},x_{j'},u+x_j^2)$ explicitly. We find
$$
Z = \Proj \mathbb Q(u)[t][x_{i'}, x_j, x_{j'}][A,B,C,D]/I
$$
where the grading comes from $A$, $B$, $C$ and $D$, and $I$ is the ideal generated by the $2\times 2$-minors of
$$
\begin{pmatrix}
A & B & C & D\\
t  & x_{i'} & x_{j'} & u+x^2_j
\end{pmatrix},
$$
together with the polynomials
$$
ut+x_{i'}^2 + tx_j^2 + ux_{j'}^2,\quad  (x_{i'}^2+ux_{j'}^2)B+tx_{i'}D, 
\quad ux_{j'}C+x_{i'}B+tD, \quad AD+B^2+uC^2. 
$$
One can now directly check that $Z \to \Spec \mathbb Q(u)$ is indeed smooth. Thus the  first statement  follows from Lemma  \ref{lem:spread}.

For the second part we note that the fibre of $Z \to \mathbb A^1_{\mathbb Q(u)}$ over $t=0$ is an affine open of $X'_{\eta_{i,j}}$. It consists of the strict transform of $X_{\eta_{i,j}}$, which we already know to consist of two conjugate components. Any other component comes from the exceptional divisor, which we can compute from the equations above to be
$$
\Proj \mathbb Q(u)[x_j][A,B,C,D]/(u+x_j^2,AD+B^2+uC^2).
$$
Hence, $X'_{\eta_{i,j}}$ has four geometrical components which are conjugated in pairs. 
\end{proof}

It follows from this result  $\delta_{D_{i,j}}(\pi')=\frac12$. Moreover, the singular locus of $X'$  does not meet the fibres $X'_{\eta_{i,j}}$ over the generic points $\eta_{i,j}$ of the $D_{i,j}$. By using a strong desingularisation, which was proven to exist by Hironaka \cite{Hironaka64}, we can conclude that $\tilde X_{\eta_{i,j}} \xrightarrow{\cong} X'_{\eta_{i,j}}$ over $Y$.
It therefore follows that 
$
\delta_{D_{i,j}}(\tilde \pi)=\frac12,
$
which thereby  completes  the proof of  \eqref{eq:goal}, via \eqref{eq:def_Delta}.

\section{Analytic  input}\label{s:ant}

In this section we prove  Theorem \ref{t}.
We note that the map $\tilde X\to X$ is an isomorphism outside the locus of $y\in \PP^3$ such that 
$y_0\dots y_3=0$.  Let $|\y|=\max_{0\leq i\leq 3}|y_i|$ if $y\in \PP^3(\Q)$  is represented by a  vector
$\y\in \ZZp^4$. The anticanonical height function on $Y(\Q)$ is then $H(y)=|\y|^2$.
Since there are $O(\sqrt{B})$ choices of $y\in \PP^3(\Q)$ of height at most $B$ for which
 $y_0\dots y_3=0$, it follows that 
$$
N_{\text{loc}}(\tilde \pi,B)=N_{\text{loc}}(\pi,B)+O(\sqrt{B}),
$$
where $\pi \colon X\to Y$.
For a given vector $\y$, let
$X_\mathbf{y}\subset\mathbb P^3$ be the diagonal quadric
$$
y_0x_0^2+y_1x_1^2+y_2x_2^2+y_3x_3^2=0.
$$
Then 
$$
N_\text{loc}(\pi,B)=\frac{1}{2}\#\left\{\y \in \ZZp^4 \colon |\y|\le \sqrt{B}, ~X_\a(\mathbf{A}_\Q)\neq\emptyset\right\}.
$$
In order to prove Theorem \ref{t}, it will  suffice to prove that 
\begin{equation}\label{eq:new}
B\ll N_\text{loc}(\pi,B)\ll B.
\end{equation}
To see the lower bound, we take
$\y=(u,v,-u,-v)$, for coprime $u,v\in \Z$ such that  $|u|,|v|\le \sqrt{B}$.
For such $\y$ we obviously have 
$(1 \colon 1 \colon 1 \colon 1)\in X_\a(\mathbf{A}_\Q).
$ 
It easily follows that 
$N_\text{loc}(\pi,B)\gg B$.

\begin{rema}
The lower bound for $N_\text{loc}(\pi,B)$ arises from points on the anti-diagonal line $D\subset Y$ defined by the equations
$y_0+y_2=y_1+y_3=0$. Consider the fibration $\pi_D \colon X_D\to D$. 
The reducible fibres occur precisely above the closed points $y_0=0$ and $y_1=0$ in $D\cong \PP^1$. 
However, the fibres above these points are split and so it follows from 
\eqref{eq:def_Delta} that 
$\Delta(\pi_D)=0$. On the other hand $\rho_D=1$ and so our lower bound is consistent with 
\eqref{eq:guess}, which would predict that 
$N_{\text{loc}}(\pi_D,B)\sim cB$ for an appropriate constant  $c>0$.
\end{rema}

It remains to prove the upper bound in \eqref{eq:new}.
Let us first dispatch the contribution to $N_\text{loc}(\pi,B)$ from $\y$ for which $y_0y_1y_2y_3=0$.
But then at least two components of $\y$ must vanish and the  fibre over any such point has an obvious rational point. There are 
$
\frac{4}{\zeta(2)}B +O(\sqrt B)
$
vectors  $(a,b)\in \ZZp^2$ with $|a|,|b|\leq B$. 

We denote by $N^{\circ}(B)$ the  contribution from 
$\y$ for which $y_0y_1y_2y_3\neq 0$. Then our work so far shows that 
\begin{equation}\label{eq:newish}
N_\text{loc}(\pi,B)=N^\circ(B)+
\frac{8}{\zeta(2)}B +O(\sqrt B).
\end{equation}
For each $\y\in\Z_{\neq 0}^4$ with $y_0y_1=y_2y_3$,  there exist $\b=(t_0,t_1,t_2,t_3)\in\Z_{\neq 0}^4$ such that 
$$
y_0=t_0t_2, \quad y_1=t_1t_3, \quad y_2=t_0t_3, \quad y_3=t_1t_2.
$$
In this way it follows that 
\begin{equation}\label{eq:circ}
N^\circ(B)\ll
\#\left\{
\b\in\Z_{\neq 0}^4\colon
\begin{array}{l}
 |t_0t_2|,|t_0t_3|,|t_1t_2|,|t_1t_3|\le \sqrt B\\ 
 X_{t_0t_2,t_1t_3,t_0t_3,t_1t_2}(\mathbf{A}_\Q)\neq\emptyset
 \end{array}
 \right\}.
\end{equation}
We shall begin by studying a related counting  function  in which only square-free integers appear, before later deducing our final estimate for 
$N^\circ(B)$.

\subsection{Square-free coefficients}

For given $
U_0,\dots,U_3>0$, let
$$
\mathcal{U}=\left\{
\u\in \Z^4:
U_i/2<|u_i|\leq U_i\text{ for $0\leq i\leq 3$}\right\}.
$$
In this section we shall be concerned with estimating 
\begin{equation}\label{eq:N*}
N^*(U_0,\dots,U_3)=
\#\left\{
\u\in \mathcal{U}\colon
\begin{array}{l}
X_{u_0u_2,u_1u_3,u_0u_3,u_1u_2}(\mathbf{A}_\Q)\neq\emptyset\\
\mu^2(u_0)\cdots \mu^2(u_3)=1
 \end{array}
 \right\},
\end{equation}
where $\mu$ is the M\"obius function.
We shall prove the upper bound
\begin{equation}\label{eq:timer}
N^*(U_0,\dots,U_3)\ll 
\frac{{U_0U_1U_2U_3}}{(1+\log U_\text{min})^2},
\end{equation}
if  $U_\text{min}=\min\{U_0,\dots,U_3\}\geq 1$. 

The upper bound 
\eqref{eq:timer} is trivial if $1\leq U_\text{min}<2$ and so we may proceed under the assumption that 
$U_\text{min}\geq 2$. Our primary tool will be the following 
multidimensional form of the large sieve inequality, as worked out by Kowalski \cite[Thm.~4.1]{large}. 

\begin{lem}\label{lem:large_sieve}
Let $$
X=\{(n_1,\ldots,n_r)\in\Z^r \colon -N_j\le n_j\le N_j \text{ for $1\leq j\leq r$}\} 
$$
and let $\Omega_p\subset \F_p^r$ for all $p$. Then
$$
\#\{x\in X \colon x\mmod p \notin\Omega_p ~\forall p\le L \}\ll \frac{\prod_{i=1}^r (N_i+L^2)}{F(L)},
$$
where $$
F(L)=\sum\limits_{n\le L}\mu^2(n)\prod\limits_{p\mid n} \dfrac{|\Omega_p|}{p^r-|\Omega_p|}.
$$
\end{lem}

Define 
 $
\Omega_p=\bigcup\limits_{i=0}^3 \Omega_{p,i},
$ 
where
\begin{gather*}
\Omega_{p,0}=\{(r_0,r_1,r_2,r_3)\in \F_p^4 \colon r_0=0,\ r_1r_2r_3\neq 0,\ -r_2r_3 \notin\F_p^{\times,2}\},\\
\Omega_{p,1}=\{(r_0,r_1,r_2,r_3)\in \F_p^4 \colon r_1=0,\ r_0r_2r_3\neq 0,\ -r_2r_3 \notin\F_p^{\times,2}\},\\
\Omega_{p,2}=\{(r_0,r_1,r_2,r_3)\in \F_p^4 \colon r_2=0,\ r_0r_1r_3\neq 0,\ -r_0r_1 \notin\F_p^{\times,2}\},\\
\Omega_{p,3}=\{(r_0,r_1,r_2,r_3)\in \F_p^4 \colon r_3=0,\ r_0r_1r_2\neq 0,\ -r_0r_1 \notin\F_p^{\times,2}\}.
\end{gather*}
Let  $\u$ be a vector  counted by 
$N^*(U_0,\dots,U_3)$. The following result shows that 
$\u \bmod{p} \not \in \Omega_p$ for any prime $p$. 

\begin{lem}\label{lem:local}
Let $p$ be a prime and let  $\u\in \Z^4$ be such that 
$\mu^2(u_0)\cdots \mu^2(u_3)=1$. If  
 $\u\bmod p\in\Omega_p$ then
$
X_{u_0u_2,u_1u_3,u_0u_3,u_1u_2}(\Q_p)=\emptyset
$.
 \end{lem}

\begin{proof}
This is standard but we include a full proof for the sake of completeness. 
Assume without loss of generality that  $\u\bmod p\in\Omega_{p,0}$. We suppose for a contradiction  that there exists 
a solution $(x_0 \colon x_1 \colon  x_2 \colon x_3)$ over $\Q_p$ to the equation
$$
u_0u_2x_0^2+u_1u_3x_1^2+u_0u_3x_2^2+u_1u_2x_3^2=0.
$$
We can assume that  $(x_0,x_1,x_2,x_3)\in\Z_p^4$ is primitive. Since $p\mid u_0$, it follows that 
 $u_1u_3x_1^2+u_1u_2x_3^2\equiv 0\bmod p$. But $p\nmid u_1$ and so
$$
u_3x_1^2+u_2x_3^2\equiv 0\bmod p.
$$ 
But  $-u_2u_3$ is a non-square modulo $p$ and so $x_1\equiv x_3\equiv 0\bmod p$. Write  $x_1=px_1'$, $x_3=px_3'$ and  $u_0=pu_0'$, 
where $p\nmid u_0'$ since 
 $u_0$ is square-free.
 Then it follows that 
$$
u_0'u_2x_0^2+pu_1u_3x_1'^2+u_0'u_3x_2^2+pu_1u_2x_3'^2=0,
$$
whence 
$u_2x_0^2+u_3x_2^2\equiv 0\bmod p$. But this implies that   $x_0\equiv x_2\equiv 0\bmod p$, which contradicts the primitivity of $(x_0,x_1,x_2,x_3)$.
\end{proof}

In order to apply the large sieve we need to calculate the size of 
$\Omega_p$. But for $p>2$ we clearly have 
$$
|\Omega_{p,i}|=(p-1)^2\cdot\frac{p-1}{2}=\frac{(p-1)^3}{2},
$$ 
for 
$0\leq i\leq 3$. Hence it follows that 
\begin{equation}\label{eq:OM}
|\Omega_p|=2p^3\left(1-\frac{1}{p}\right)^3.
\end{equation}
Lemma \ref{lem:local} implies that 
$$
N^*(U_0,\dots,U_3)\leq 
\#\left\{\u\in \mathcal{U}\colon
\u\bmod{p}\not\in \Omega_p ~\forall p\leq L\right\},
 $$
for any $L\geq 1$. On appealing to 
Lemma \ref{lem:large_sieve}, we conclude
that 
\begin{equation}\label{eq:time}
N^*(U_0,\dots,U_3)\ll 
\frac{\prod_{i=0}^3 (U_i +L^2)}{F(L)},
\end{equation}
where  
 \begin{align*}
F(L)
&=\sum\limits_{n\le L}\mu^2(n)\prod\limits_{p\mid n} \dfrac{|\Omega_p|}{p^4-|\Omega_p|}.
\end{align*}
The following result gives a lower bound for this quantity.

\begin{lem}\label{lem:log}
We have $F(L)\gg (\log L)^2$, for any $L\geq 2$.
\end{lem}

\begin{proof}
It follows from \eqref{eq:OM} that 
\begin{align*}
\dfrac{|\Omega_p|}{p^4-|\Omega_p|}
&= \frac{2p^3(1-\frac{1}{p})^3}{p^4(1-\frac{2}{p}(1-\frac{1}{p})^3)}\\
&\geq \frac{2}{p}\left(1-\frac{1}{p}\right)^3.
\end{align*}
Hence 
\begin{align*}
F(L)
&\geq   \sum_{n\le L}\frac{\mu^2(n)\tau(n)\phi^*(n)^3}{n},
\end{align*}
where $\tau(n)$ is the divisor function and 
$$
\phi^*(n)=\prod_{p\mid n}\left(1-\frac{1}{p}\right).
$$
We now  apply a standard result on the asymptotic evaluation of multiplicative arithmetic functions supported on square-free integers, as supplied  by Friedlander and Iwaniec \cite[Thm.~A.5]{FI}.
Let $g(n)=\mu^2(n) \tau(n)\phi^*(n)^3/n$. Then 
$$
\sum_{p\leq x}g(p)\log p=2\sum_{p\leq x}\frac{\log p}{p} +O(1)=2\log x +O(1),
$$
by Mertens' theorem.  Moreover, the conditions (A.16) and (A.17) of  
 \cite[Thm.~A.5]{FI} are easily verified. Hence, on taking $k=2$ in this result, 
  it follows that 
 $$
\sum_{n\leq L} g(n) \gg (\log L)^2,
 $$
 which thereby establishes the lemma. 
\end{proof}

Let $U_\text{min}=\min\{U_0,\dots,U_3\}$ and recall that we are assuming that  $U_\text{min}\geq 2$. 
Inserting Lemma \ref{lem:log} into \eqref{eq:time} and taking $L=\sqrt{U_\text{min}}$, we finally arrive at the upper bound \eqref{eq:timer}.

\subsection{Deduction of the general case}

We now turn to the estimation of $N^\circ(B)$, starting from \eqref{eq:circ}.
Pick $T_0,\dots,T_3$ such that 
\begin{equation}\label{eq:Ti}
T_0,\dots,T_3\geq 1 \quad \text{ and }\quad 
T_0T_2 , T_0T_3, T_1T_2, T_1T_3\leq 4\sqrt B.
\end{equation}
We shall examine the contribution $N(T_0,\dots,T_3)$, say, from those $\b\in\Z_{\neq 0}^4$ for which 
$T_i/2<|t_i|\leq T_i$, for $0\leq i\leq 3$.
Then 
\begin{equation}\label{eq:timeless}
N^\circ(B)\ll \sum_{T_0,\dots,T_3} N(T_0,\dots,T_3),
\end{equation}
where $T_0,\dots,T_3$ run over powers of $2$ subject to the inequalities in \eqref{eq:Ti}.

Each  non-zero integer $t_i$ admits a factorisation $u_im_i^2$ for square-free $u_i\in \Z$ and $m_i\in \Z_{>0}$.
Moreover, 
$$
 X_{t_0t_2,t_1t_3,t_0t_3,t_1t_2}(\mathbf{A}_\Q)\neq\emptyset \Longleftrightarrow
 X_{u_0u_2,u_1u_3,u_0u_3,u_1u_2}(\mathbf{A}_\Q)\neq\emptyset .
 $$
 Hence it follows that 
 $$
 N(T_0,\dots,T_3)\leq \sum_{m_0\leq \sqrt{T_0}}
 \dots 
 \sum_{m_3\leq \sqrt{T_3}} N^*\left(\frac{T_0}{m_0^2},\dots,\frac{T_3}{m_3^2}\right),
 $$
 in the notation of 
\eqref{eq:N*}. On appealing to \eqref{eq:timer}, we deduce that 
 \begin{align*}
 N(T_0,\dots,T_3)
 &\ll 
 \sum_{m_0\leq \sqrt{T_0}}
 \dots 
 \sum_{m_3\leq \sqrt{T_3}} 
  \sum_{j\in \{0,\dots,3\}}
\frac{T_0T_1T_2T_3}{m_0^2m_1^2m_2^2m_3^2 (1+\log (T_j/m_j^2))^2 }
\\
&\ll 
T_0T_1T_2T_3
 \sum_{j\in \{0,\dots,3\}}
 \sum_{m\leq \sqrt{T_j}}
\frac{1}{m^2 (1+\log (T_j/m^2))^2 }\\
&\ll 
\frac{T_0T_1T_2T_3}{(1+\log T_\text{min})^2},
 \end{align*}
 where $T_\text{min}=\min\{T_0,\dots,T_3\}$.
Returning to 
\eqref{eq:timeless}, we finally conclude that 
$$
N^\circ(B)\ll \sum_{T_0,\dots,T_3} \frac{T_0T_1T_2T_3}{(1+\log T_\text{min})^2},
$$
where $T_0,\dots,T_3$ run over powers of $2$ subject to the inequalities in \eqref{eq:Ti}.

By symmetry, we can suppose without loss of generality that 
$T_0\leq T_1$ and $T_2\leq T_3$.
Summing first over $T_1\leq 4\sqrt{B}/T_3$,  we see that
\begin{align*}
\sum_{T_0,\dots,T_3} \frac{T_0T_1T_2T_3}{(1+\log T_\text{min})^2}
&\ll \sqrt B \sum_{T_0,T_2,T_3}
\frac{T_0T_2}{(1+\log \min\{T_0,T_2\})^2}.
\end{align*}
Suppose first that $T_0\leq T_2$. Then we sum over $T_2\leq T_3$ and then $T_3\leq 4\sqrt{B}/T_0$ to get
\begin{align*}
\sum_{T_0,\dots,T_3} \frac{T_0T_1T_2T_3}{(1+\log T_\text{min})^2}
&\ll \sqrt B \sum_{T_0,T_3}
\frac{T_0T_3}{(1+\log T_0)^2}\\
&\ll B \sum_{T_0}
\frac{1}{(1+\log T_0)^2},
\end{align*}
where the sum is over $T_0=2^j$ with $1\leq 2^j\leq 2B^{1/4}$.  This sum is convergent, whence this case contributes  $O(B)$. Alternatively, if 
 $T_0> T_2$, then we sum over $T_0\leq 4\sqrt{B}/T_3$  and then $T_3\geq T_2$ 
to get
\begin{align*}
\sum_{T_0,\dots,T_3} \frac{T_0T_1T_2T_3}{(1+\log T_\text{min})^2}
&\ll \sqrt B \sum_{T_0,T_2,T_3}
\frac{T_0T_2}{(1+\log T_2)^2}\\
&\ll B \sum_{T_2,T_3}
\frac{T_2}{T_3(1+\log T_2)^2}\\
&\ll B \sum_{T_2}
\frac{1}{(1+\log T_2)^2}.
\end{align*}
This also makes the satisfactory contribution $O(B).$

Hence we have shown that $N^\circ(B)=O(B).$
Once inserted into \eqref{eq:newish}, this therefore completes the proof of the upper bound in \eqref{eq:new}.


\begin{thebibliography}{99}

\bibitem{BL}
T.D. Browning and D. Loughran, 
Sieving rational points on varieties.
{\em Trans. Amer. Math. Soc.} {\bf 371} (2019), 5757--5785.

\bibitem{f-m-t}
J. Franke, Y.I. Manin and Y. Tschinkel,
Rational points of bounded height on {F}ano varieties.
{\em Invent.\ Math.\ } {\bf 95} (1989), 421--435.

\bibitem{FI}
J.B. Friedlander and  H. Iwaniec,
 {\em Opera de cribro}. American Math.\ Soc., 2010.

\bibitem{fi}
J.B. Friedlander and H. Iwaniec,
Ternary quadratic forms with rational zeros.
{\em J.\ Th\'eor.\ Nombres Bordeaux} {\bf 22} (2010), 97--113.

\bibitem{Hironaka64}
H. Hironaka,
Resolution of singularities of an algebraic variety over a field of characteristic zero. {I \& II}.
{\em Annals  Math.} {\bf 79} (1964), 109--203 \& 205--326.

\bibitem{large}
E. Kowalski, {\em The large sieve and its applications}.
Cambridge University Press,  2009.
	\bibitem{LS} 
D. Loughran and A. Smeets, Fibrations with few rational points. 
{Geom. Funct. Anal.} {\bf 26} (2016), 1449--1482.

\bibitem{peyre-freedom}
E. Peyre, 
Libert\'e et accumulation. {\em 
Documenta Math.} {\bf 22} (2017), 1615--1659.


	\bibitem{Poonen}
B. Poonen, {\em Rational Points on Varieties}.
Volume 186 of {\em Graduate Studies in Mathematics}.
American Mathematical Society, Providence, RI, 2017.
	
	\bibitem{PV}
B. Poonen and J.F. Voloch, Random Diophantine equations. 
{\em Arithmetic of higher-dimensional algebraic varieties (Palo Alto,
CA, 2002)}, 175--184,
Progr. Math. {\bf 226}, Birkh\"auser, 2004.



\end{thebibliography}
\end{document}